\documentclass[11pt, a4paper]{amsart}
\usepackage{amssymb, array, amsmath, amscd, pdfpages, enumerate, amsthm, fixltx2e, setspace, hyperref, bbm,overpic}
\usepackage[justification=centering,margin=0.5cm]{caption}
\usepackage[T1]{fontenc}

\DeclareMathOperator*{\Ran}{Ran}
\DeclareMathOperator*{\Ker}{Ker}

\DeclareMathOperator*{\R}{Re}

\newcommand{\dd}{\mathrm{d}}
\newcommand{\B}{\mathcal{B}}

\newcommand{\RR}{\mathbb{R}}
\newcommand{\CC}{\mathbb{C}}

\newcommand{\NN}{\mathbb{N}}
\newcommand{\ZZ}{\mathbb{Z}}
\newcommand{\T}{(T(t))_{t\ge 0}}
\newcommand{\TN}{(T_N(t))_{t\ge 0}}
\newcommand{\eps}{\varepsilon}

\newtheorem{thm}{Theorem}[section]
\newtheorem{prp}[thm]{Proposition}
\newtheorem{lem}[thm]{Lemma}

\theoremstyle{definition}
\newtheorem{rem}[thm]{Remark}
\newtheorem{ex}[thm]{Example}

\numberwithin{equation}{section}

\begin{document}
\title[Asymptotics and Approximation of Large Systems of ODEs]{ Asymptotics and Approximation of Large Systems of Ordinary Differential Equations}

\author[L. Paunonen]{Lassi Paunonen}
\address[L. Paunonen]{Department of Mathematics, Tampere University of Technology, PO.\ Box 553, 33101 Tampere, Finland}
\email{lassi.paunonen@tut.fi}

\author[D. Seifert]{David Seifert}
\address[D. Seifert]{School of Mathematics, Statistics and Physics, Herschel Building, Newcastle University, Newcastle upon Tyne, NE1 7RU, UK}
\email{david.seifert@ncl.ac.uk}

\begin{abstract}
In this paper we continue our earlier investigations into the asymptotic behaviour of infinite systems of coupled differential equations. Under the mild assumption that the so-called characteristic function of our system is completely monotonic we obtain a drastically simplified condition which ensures boundedness of the associates semigroup. If the characteristic function satisfies certain additional conditions we deduce sharp rates of convergence to equilibrium. We moreover address the important and delicate issue of the role of the infinite system in understanding the asymptotic behaviour of large but finite systems, and we provide a precise way of obtaining size-independent rates of convergence for families of finite-dimensional systems. Finally, we illustrate our abstract results in the setting of the well-known platoon problem.

\end{abstract}

\subjclass[2010]{93D20, 34D05, 47A58  (34H15, 47D06, 47A10).}
\keywords{Coupled differential equations, asymptotic behaviour, rates of convergence, $C_0$-semigroup, spectrum, completely monotonic function, approximation.}
\thanks{L.\ Paunonen is supported by the Academy of Finland grants 298182 and 310489.}

\maketitle

\section{Introduction}\label{sec:intro} 

In this paper we continue and expand on our earlier investigations \cite{PauSei17a, PauSei18} of infinite coupled systems of ordinary equations of the form
\begin{equation}\label{eq:sys_two}
\dot{x}_k(t)=A_0x_k+A_1x_{k-1}(t), \quad t>0,\  k\in\ZZ,
\end{equation}
subject to initial conditions for $x_k(0)$, $k\in\ZZ$, and we also consider one-sided systems of the form
\begin{equation}\label{eq:sys_one}
\left\{\begin{aligned}
\dot{x}_1(t)&=A_0x_1(t), &&t>0,\\
\dot{x}_k(t)&=A_0x_k(t)+A_1x_{k-1}(t),\quad&&  t>0,\ k\ge2,
\end{aligned}\right.
\end{equation}
subject to initial conditions for $x_k(0)$, $k\in\NN$. 
 Here the unknown functions $x_k$ take values in $\CC^m$ for some $m\in\NN$ and $A_0,\ A_1$ are $m\times m$ matrices. In the one-sided case \eqref{eq:sys_one}, the dynamics of the agent corresponding to $k=1$ is independent of the others, so this agent can be thought of as the ``leader''. Indeed, systems of this type, which are sometimes referred to as \emph{spatially invariant systems}, arise naturally in so-called \emph{platoon problems} in control theory, where each $x_k$ describes the state of an agent in an infinite vehicle platoon, and the aim is to choose the matrices $A_0$, $A_1$ in such a way that certain control objectives are met in the large-time limit $t\to\infty$; see \cite{ Cur15, CurIft09, PloSch11, SwaHed96,ZwaFir13}. We mention that most of the early work in this area is restricted to the $\ell^2$-setting, in which one may use certain results from Fourier analysis. A simple but important special case of a platoon model arises in the so-called \emph{robot rendezvous problem} \cite{FeiFra12,FeiFra12b, PauSei17b}, and it was in the context of this problem that the authors of \cite{FeiFra12,FeiFra12b} originally argued for the need to study more general $\ell^p$-norms, as was done in \cite{PauSei17a, PauSei18} and as will be done here.
 
Infinite systems of this type have been used as approximations to large but finite systems, for instance in \cite{JovBam05}. However, the precise relationship between the dynamical properties of the infinite system and those of  its truncated versions is a highly delicate matter, as has been pointed out for instance by Ruth Curtain and her collaborators in~\cite{Cur12, Cur15, CuIfZw09}. In particular, each finite truncation of an infinite system will typically be uniformly exponentially stable even when the corresponding infinite system is not. As one of our main new contributions we present rates of convergence for large finite systems which are uniform with respect to the size. More specifically, for a class of finite systems which are exponentially stable but without a uniform stability margin we obtain optimal subexponential rates of convergence which are independent of the size of the system.
   
As in our previous work \cite{PauSei17a, PauSei18} we impose certain natural standing assumptions in order to make the model amenable to  the techniques from semigroup theory. Most importantly, we assume throughout that $A_1\ne0$ and  that there exists a rational function $\phi$ such that
\begin{equation}\label{eq:char}
A_1R(\lambda,A_0)A_1=\phi(\lambda)A_1,\quad \lambda\in\rho(A_0).
\end{equation}
As discussed in \cite{PauSei17a},  in this case the set of poles of $\phi$ is contained (perhaps as a strict subset) in the set of eigenvalues $\sigma(A_0)$ of $A_0$, and moreover $|\phi(\lambda)|\to0$ as $|\lambda|\to\infty$. We call the function $\phi$ the \emph{characteristic function} of our system. Functions admitting a characteristic functions arise in numerous applications. For instance, a characteristic function may be found whenever $A_1$ has rank 1.

In Section~\ref{sec:infinite} we consider the abstract Cauchy problem associated with our infinite system. We focus here on the one-sided case, which was not treated in \cite{PauSei17a, PauSei18}. Building on our earlier work, we first develop a detailed spectral theory of the system generator, before presenting as one of our first main results, Theorem~\ref{thm:bdd}, a new and drastically simplified sufficient condition for the associated semigroup to be bounded in the important special case where the characteristic function $\phi$ is \emph{completely monotonic}. In Theorem~\ref{thm:asymp} we present a detailed description of the quantified asymptotic behaviour of the infinite one-sided system, and as another of our main results we show in Theorem~\ref{thm:no_log} that one obtains an even sharper (and indeed optimal) rate of convergence to equilibrium under certain additional assumptions on $\phi$. We illustrate our abstract results by providing examples in which $\phi$ satisfies the different hypotheses, and we moreover illustrate how the improvements in the present paper over \cite{PauSei17a, PauSei18} carry over to the two-sided case considered there. Then, in Section~\ref{sec:finite}, we turn to the important question of the relation between infinite systems and large but finite ones.  We show in Theorem~\ref{thm:uniform} that there is a precise sense in which an understanding of the infinite system can lead directly to rates of decay of truncated systems which are uniform with respect to the size of the finite system. This important new result follows from a simple but powerful abstract result about families of semigroups.  Finally, in Section~\ref{sec:platoon}, we illustrate the power of our main results by applying them to the concrete platoon model studied also in \cite{PauSei17a, PauSei18}.

Our notation is standard and largely follows that used in \cite{PauSei17a, PauSei18}. In particular, we write $X^*$ for the dual space of a complex Banach space $X$, and we write $\Ker A$ for the kernel and $\Ran A$ for the range of an element $A$ of $\B(X)$, the set of all bounded linear operators on $X$. Moreover, we let $\sigma(A)$ denote the spectrum of $A\in\B(X)$ and for $\lambda$ in the resolvent set $\rho(A)=\CC\setminus\sigma(A)$ we write $R(\lambda,A)$ for the resolvent operator $(\lambda-A)^{-1}$. We write $\sigma_p(A)$ for the point spectrum of $A$. Given $A\in\B(X)$ we denote the dual operator of $A$ by $A^*$, while the transpose of a matrix $A$ is denoted by $A^T$. We use standard asymptotic notation such as `big O', and we write $f(\lambda)\asymp g(\lambda)$ as $\lambda\to\lambda_0$ if both $f(\lambda)=O(g(\lambda))$ and $g(\lambda)=O(f(\lambda))$ as $\lambda\to\lambda_0$. Finally, we denote the open right/left half-plane by $\CC_\pm=\{\lambda\in\CC:\R\lambda\gtrless 0\}$, and we use a horizontal bar over a set to denote its closure.

\section{Infinite systems of differential equations}\label{sec:infinite}

We focus here on the one-sided system \eqref{eq:sys_one}, which was not treated in our earlier works \cite{PauSei17a, PauSei18}. We begin by reformulating it as an abstract Cauchy problem of the form 
\begin{equation}\label{eq:ACP}
\left\{\begin{aligned}
\dot{x}(t)&=Ax(t),\quad t\ge0,\\
{x}(0)&=x_0\in X,
\end{aligned}\right.
\end{equation}
where $X=\ell^p(\NN,\CC^m)$ for $1\le p\le\infty$ and 
\begin{equation}\label{eq:A}
Ax=(A_0x_1,A_0x_2+A_1x_1,A_0x_3+A_1x_2,\dotsc),\quad x=(x_k)_{k\ge1}\in X,
\end{equation}
so that the system operator $A$ is a bounded linear operator on $X$. Here and in what follows we endow $X=\ell^p(\NN,\CC^m)$ with its natural Banach space norm, using the Euclidean norm on $\CC^m$. In particular, $X$ is a Hilbert space if and only if $p=2$.

\subsection{Spectral properties of the system operator}\label{sec:spec}

We begin by investigating the spectrum of the operator $A$ defined in \eqref{eq:A}. The following result may be viewed as an analogue in the one-sided setting of \cite[Theorem~2.3]{PauSei17a}. Here and in what follows we let 
\begin{equation}\label{eq:Omega}
\Omega_\phi^+=\{\lambda\in\rho(A_0):|\phi(\lambda)|\ge1\}.
\end{equation}
Later on, when we consider the two-sided case, we shall also consider the set $\Omega_\phi=\{\lambda\in\rho(A_0):|\phi(\lambda)|=1\}.$

\begin{prp}\label{prp:spec}
Let $m\in\NN$ and $1\le p\le\infty$, and consider the operator $A$ on $X$ defined in \eqref{eq:A}.  Then 
$\sigma(A)=\sigma(A_0)\cup\Omega_\phi^+$ and $\sigma_p(A)\subseteq\sigma(A_0).$ Moreover, $\Ran(\lambda-A)\ne X$ for all $\lambda\in\sigma(A)$, and $\Ran(\lambda-A)$ is dense in $X$ if and only if $\lambda\in\rho(A_0)$, $|\phi(\lambda)|=1$ and $1<p<\infty$.
\end{prp}

\begin{proof}
If $\lambda\in\rho(A_0)$ and $|\phi(\lambda)|<1$, then it is straightforward to verify that the operator $\lambda-A$ is invertible and that its inverse is given for $x\in X$ by 
\begin{equation}\label{eq:res}
R(\lambda,A)x=\left(R_\lambda x_k+R_\lambda A_1R_\lambda\sum_{\ell=0}^{k-2}\phi(\lambda)^\ell x_{k-\ell-1}\right)_{k\ge1},
\end{equation}
where $R_\lambda=R(\lambda,A_0)$ for $\lambda\in\rho(A_0)$. Hence $\sigma(A)\subseteq\sigma(A_0)\cup\Omega_\phi^+$. Further, it is straightforward to verify that $\lambda-A$ is injective whenever $\lambda\in\rho(A_0)$, so $\sigma_p(A)\subseteq\sigma(A_0)$. We now prove the statements about the range of $\lambda-A$ for $\lambda\in\sigma(A_0)\cup\Omega_\phi^+$, which will complete the proof. First, if $\lambda\in\sigma(A_0)$ and if $y_1\in\Ker(\lambda-A_0^T)\setminus\{0\}$ then the sequence $(y_1,0,0,\dots)$, interpreted as an element of the dual space $X^*$, lies in $\Ker(\lambda-A^*)$, and hence $\Ran(\lambda-A)$ cannot be dense in $X$. Suppose now that $\lambda\in\rho(A_0).$ If $|\phi(\lambda)|>1$ and if $y_0\in\CC^m$ is such that $R_\lambda^TA_1^Ty_0\ne0$ then the sequence $(\phi(\lambda)^{-k}R_\lambda^TA_1^Ty_0)_{k\ge1}$, again interpreted as an element of $X^*$, is easily seen to lie in $\Ker(\lambda-A^*)$, so $\Ran(\lambda-A)$ again fails to be dense in $X$. If $|\phi(\lambda)|=1$ the same argument carries over to the case $p=1$, while for $p=\infty$ we may proceed as in the proof of \cite[Theorem~2.3]{PauSei17a} to obtain the same conclusion. It remains to consider the case in which $|\phi(\lambda)|=1$ and $1<p<\infty$. We shall show that in this case $\Ran(\lambda-A)$ is a proper dense subspace of $X$. To see that $\Ran(\lambda-A)\ne X$ it suffices to observe that if $y_1\in\CC^m$ is such that $A_1R_\lambda y_1\neq 0$ then the sequence $(y_1,0,0,\dots)\in X$ does not lie in the range of $\lambda-A$. Let us identify $X^*$ with $\ell^q(\NN,\CC^m)$, where $q\in(1,\infty)$ is the H\"older conjugate of $p$ and suppose that $y=(y_k)_{k\ge1}\in \Ker(\lambda-A^*)$. Then $(\lambda-A_0^T)y_k=A_1^Ty_{k+1}$ and a simple calculation yields $R_\lambda^TA_1^Ty_k=\phi(\lambda)^{1-k}R_\lambda^TA_1^Ty_1$, $k\ge1$. Hence the vectors $z_k=y_k-\phi(\lambda)^{1-k}y_1$, $k\ge1$, lie in $\Ker A_1^T$, and we have
\begin{equation}\label{eq:z}
(\lambda-A_0^T)(z_k+\phi(\lambda)^{1-k}y_1)=\phi(\lambda)^{-k}A_1^Ty_1, \quad k\ge1.
\end{equation}
For $k=1$ this becomes $(\lambda-A_0^T)y_1=\phi(\lambda)^{-1}A_1^Ty_1$, and substituting this back into \eqref{eq:z} shows that $z_k=0$ for all $k\ge1$. Hence $y_k=\phi(\lambda)^{1-k}y_1$, $k\ge1$, so in order to have $y\in\ell^q(\NN,\CC^m)$ we must have $y=0$. Thus $\lambda-A^*$ is injective, and hence $\Ran(\lambda-A)$ is dense in $X$, as required.
\end{proof}

\begin{rem}
Note that $\sigma_p(A)$ may be strictly contained in $\sigma(A_0)$. For instance, it is easy to see that if $\lambda\in\sigma(A_0)$ is such that $A_1\Ker(\lambda-A_0)\not\subseteq\Ran(\lambda-A_0)$ then $\lambda\not\in\sigma_p(A)$. On the other hand if $\lambda\in\sigma(A_0)$ and $\Ker(\lambda-A_0)\cap\Ker A_1\ne\{0\}$ then necessarily $\lambda\in\sigma_p(A)$.
\end{rem}

We now establish an important resolvent bound for the operator $A$. The following result is a one-sided version of \cite[Proposition~2.5]{PauSei17a}. The proof is the same, except that the formula for the resolvent in the two-sided case is replaced by  \eqref{eq:res}; we therefore omit it.

\begin{prp}\label{prp:res}
Let $m\in\NN$ and $1\le p\le\infty$, and consider the operator $A$ on $X$ defined in \eqref{eq:A}. Let $\lambda\in\rho(A_0)$ be such that $|\phi(\lambda)|<1$. Then $\lambda\in\rho(A)$ and
$$\bigg|\|R(\lambda,A)\|-\frac{\|R(\lambda,A_0)A_1R(\lambda,A_0)\|}{1-|\phi(\lambda)|}\bigg|\le \|R(\lambda,A_0)\|.$$
In particular, for $\lambda_0\in\rho(A_0)$ satisfying $|\phi(\lambda_0)|=1$ we have 
$$\|R(\lambda,A)\|\asymp\frac{1}{1-|\phi(\lambda)|}$$
as $\lambda\to\lambda_0$ in the region $\{\lambda\in\rho(A_0):|\phi(\lambda)|<1\}$.
\end{prp}

We continue this section with the following important observation. The proof is the same as that of \cite[Lemma~2.6]{PauSei17a} and is therefore omitted.

\begin{lem}\label{lem:res}
Let $m\in\NN$ and $1\le p\le\infty$, and assume that $\sigma(A_0)\subseteq\CC_-$ and  that the set $\Omega_\phi^+$ defined in \eqref{eq:Omega} is such that $0\in\Omega_\phi^+\subseteq\CC_-\cup\{0\}$. Then $i\RR\setminus\{0\}\subseteq\rho(A)$ and there exists an even integer $n\ge2$ such that $1-|\phi(is)|\asymp|s|^n$ as $|s|\to0$.
\end{lem}

We call the integer $n\ge2$ appearing in Lemma~\ref{lem:res} the \emph{resolvent growth parameter} associated with our system. In the remainder of the paper we shall frequently assume that the characteristic function is completely monotonic. Recall that a smooth function $f\colon(0,\infty)\to\RR$ is said to be \emph{completely monotonic} if
$(-1)^nf^{(n)}(s)\ge0$ for all $s>0$ and all integers $n\ge0$. Note that the class of all completely monotonic functions is closed under taking sums, products and under multiplication by non-negative scalars. By Bernstein's theorem \cite{Ber28} a function is completely monotonic if and only if it is the Laplace transform of a non-negative finite Borel measure. We shall say that the characteristic function $\phi$ of our system is completely monotonic if $(0,\infty)\subseteq\rho(A_0)$ and the restriction of $\phi$ to $(0,\infty)$ is real-valued and completely monotonic. Note that whenever $\sigma(A_0)\subseteq \CC_-$ then, by decomposing into partial fractions, we see that the rational function $\phi$ may be written as the Laplace transform of a non-zero smooth function $g\colon(0,\infty)\to\CC$ which is a linear combination of terms having the form $t\mapsto t^ne^{\xi t}$, $t\ge0$, for some $n\in\ZZ_+$ and some $\xi\in\CC_-$, and in particular the functions $t\mapsto t^ng(t)$ are integrable over $(0,\infty)$ and $\phi^{(n)}(0)=(-1)^n\int_0^\infty t^ng(t)\,\dd t$. If $\phi$ in addition is completely monotonic, the function $g$ is non-negative and we have  $\phi^{(n)}(0)\ne0$ for all $n\in\NN$. The following result complements Lemma~\ref{lem:res}. The statement concerning completely monotonic characteristic functions will be useful throughout the remainder of the paper.

\begin{lem}\label{lem:res2}
Suppose that $\sigma(A_0)\subseteq\CC_-$, that $0\in\Omega_\phi^+\subseteq\CC_-\cup\{0\}$ and that the restriction of the characteristic function $\phi$ to $(0,\infty)$ is real-valued. Then the resolvent growth parameter $n_\phi$  satisfies $n_\phi=2$ if and only if $\phi''(0)\ne\phi'(0)^2$. In particular,  $n_\phi=2$ whenever $\phi$ is completely monotonic.
\end{lem}

\begin{proof}
The real-valuedness assumption ensures that $\phi^{(n)}(0)$ is real for all $n\ge0$, and  by continuity of $\phi$ we obtain $\phi(0)=1$. A simple Taylor expansion about $s=0$ shows that
 $$|\phi(is)|-1=\frac12(\phi'(0)^2-\phi''(0))s^2+O(|s|^3),\quad |s|\to0,$$
 which gives the first claim. If $\phi$ is completely monotonic, let $g\colon(0,\infty)\to\RR$ be the non-negative function whose Laplace transform is $\phi$. By the first part, if $n_\phi\ne2$ then 
 $$\int_0^\infty tg(t)\,\dd t=\left(\int_0^\infty t^2g(t)\,\dd t\right)^{1/2}.$$
 Now $\phi(0)=\int_0^\infty g(t)\,\dd t=1$, so we have equality in the Cauchy-Schwarz inequality applied in $L^2(0,\infty)$ to the functions $t\mapsto g(t)^{1/2}$ and $t\mapsto tg(t)^{1/2}$. Since the two functions are non-collinear we obtain the required contradiction, and hence $n_\phi=2$.
\end{proof}

\begin{rem}
By considering Taylor expansions of higher order one could similarly derive conditions on $\phi$ which ensure that $n_\phi=4$, $n_\phi=6,$ etc. We focus here only on the case $n_\phi=2$ since in what follows we will primarily be interested in characteristic functions which are completely monotonic.
\end{rem}

\subsection{Boundedness of the semigroup.}\label{sec:mon}

In this section we provide a drastically simpler sufficient condition than in \cite{PauSei17a} for the $C_0$-semigroup $\T$ generated by the system operator defined \eqref{eq:A} to be bounded, which is to say that $\sup_{t\ge0}\|T(t)\|<\infty$. Our sufficient condition relies crucially on the notion of complete monotonicity introduced in the previous section. 

\begin{thm}\label{thm:bdd}
Let $m\in\NN$ and $1\le p\le \infty$, and suppose that $\sigma(A_0)\subseteq\CC_-$ and $\Omega_\phi^+\subseteq\overline{\CC_-}$. Suppose further that the characteristic function $\phi$ is completely monotonic. 
Then the semigroup $\T$ generated by  $A$ is bounded.
\end{thm}

\begin{proof}
The proof of \cite[Theorem~3.1]{PauSei17a}, with only trivial modifications to pass from the two-sided to the one-sided case, shows that the semigroup $\T$ is bounded provided that  
\begin{equation}\label{eq:phi_bound}
\sup_{0<\lambda\le1}\frac{\lambda}{1-|\phi(\lambda)|}<\infty.
\end{equation}
 and 
\begin{equation}\label{eq:sum_bound}
\sup_{n\in\NN}\sup_{\lambda>0}\frac{\lambda^{n+1}}{n!}\sum_{\ell=1}^\infty\bigg|\frac{\dd^n}{\dd\lambda^n}\phi(\lambda)^\ell\bigg|<\infty.
\end{equation}
Note that in the condition given in \cite{PauSei17a} the sum in \eqref{eq:sum_bound} begins at $\ell=0$, which makes no difference, however, as this additional constant term disappears after differentiating. Since $\phi$ is assumed to be completely monotonic, the discussion preceding Lemma~\ref{lem:res2} shows that $\phi'(0)\ne0$. From this  \eqref{eq:phi_bound} follows using a simple Taylor expansion argument. We focus now on establishing \eqref{eq:sum_bound}. Since $\phi$ is assumed to be completely monotonic, so is the restriction of $\phi(\cdot)^\ell$ to $(0,\infty)$ for every $\ell\ge1$. In particular, the sign of the derivatives in \eqref{eq:sum_bound} depends only on $n$, which allows us to take the modulus outside the sum. What we need to show is that
\begin{equation}\label{eq:mon_bd}
\sup_{n\in\NN}\sup_{\lambda>0}\frac{\lambda^{n+1}}{n!}\bigg|\frac{\dd^n}{\dd\lambda^n}\frac{\phi(\lambda)}{1-\phi(\lambda)}\bigg|<\infty.
\end{equation}
Now the rational function $\lambda\mapsto\phi(\lambda)(1-\phi(\lambda))^{-1}$ has no poles in $\CC_+$, and  its only possible pole on the imaginary axis is at $\lambda=0$, which by \eqref{eq:phi_bound} has order at most one. Thus we may find a constant $c\in\CC$ such that
$$\frac{\phi(\lambda)}{1-\phi(\lambda)}=\frac{c}{\lambda}+\phi_0(\lambda),$$
where  $\phi_0$ is a linear combination of functions of the form $\lambda\mapsto(\lambda-\xi)^{-k}$ for $\xi\in\CC_-$, $k\in\NN$. On the open right half-plane the latter function agrees with the Laplace transform of the bounded function $t\mapsto \smash{\frac{1}{(k-1)!}t^{k-1}e^{\xi t}}$, $t\ge0$. It follows for instance from Widder's Theorem \cite{Wid36} that 
$$\sup_{n\in\NN}\sup_{\lambda>0}\frac{\lambda^{n+1}}{n!}\bigg|\frac{\dd^n}{\dd\lambda^n}\frac{\phi(\lambda)}{1-\phi(\lambda)}\bigg|\le|c|+\sup_{n\in\NN}\sup_{\lambda>0}\frac{\lambda^{n+1}}{n!}\bigg|\frac{\dd^n\phi_0(\lambda)}{\dd\lambda^n}\bigg|<\infty.$$
Hence \eqref{eq:sum_bound} holds and the proof is complete.
\end{proof}

\begin{rem}
As pointed out by a referee, 
in the last part of the proof of Theorem 2.7
 property~\eqref{eq:mon_bd} can alternatively be verified using a minimal realisation for the rational transfer function $\lambda\mapsto \phi(\lambda)(1-\phi(\lambda))^{-1}$.
\end{rem}

\begin{ex}\phantomsection\label{ex:comp_mon}
\begin{enumerate}[(a)]
\item If $\phi(\lambda)=\zeta(\lambda+\zeta)^{-1}$ for some $\zeta>0$ then $\phi$ is completely monotonic. Since products of completely monotonic functions are completely monotonic also functions of the form
\begin{equation}\label{eq:phi_prod}\phi(\lambda)=\frac{\zeta_1\cdots\zeta_m}{(\lambda+\zeta_1)\cdots(\lambda+\zeta_m)}
\end{equation}
for some $\zeta_1,\dots,\zeta_m>0$ are completely monotonic.
\item Consider the function
\begin{equation}\label{eq:phi_pair}\phi(\lambda)=\frac{(a^2+b^2)c}{(\lambda+c)(\lambda^2+2a\lambda+a^2+b^2)},
\end{equation}
where $a,b,c>0$. Then $\phi$ has poles at $\lambda=-c$ and $\lambda=-a\pm ib$. Moreover, $\phi$ is the Laplace transform of the function $g$ defined by
$$g(t)=\frac{(a^2+b^2)c}{(a-c)^2+b^2}\left(e^{-ct}+\left(\frac{c-a}{b}\sin(bt)-\cos(bt)\right)e^{-at}\right), $$
and a simple argument shows that $g(t)\ge0$ for all $t\ge0$ if and only if $a\ge c$. Hence, by Bernstein's theorem, $\phi$ is completely monotonic if and only if $a\ge c$.
\item By combining the examples given in (a) and (b) we obtain complete monotonicity for many other rational functions whose sets of poles consist of negative real numbers and complex conjugate pairs with negative real parts.
\end{enumerate}
\end{ex}

\subsection{Asymptotic behaviour}\label{sec:asymp}

We now study the asymptotic behaviour as $t\to\infty$ of the solutions $x(t)=T(t)x_0$, $t\ge0$, of the abstract Cauchy problem \eqref{eq:ACP}. The key ingredient needed in the proof of Theorem~\ref{thm:asymp} is the following asymptotic result, which combines  special cases of \cite[Theorem~8.4]{BatChi16} and \cite[Proposition~3.1]{Mar11}; see also \cite{RSS19}.

\begin{thm}\label{thm:mlog} Let $X$ be a complex Banach space and let $\T$ be a bounded $C_0$-semigroup on $X$ whose generator $A$ is bounded and satisfies $\sigma(A)\cap i\RR=\{0\}$. Suppose moreover that 
\begin{equation}\label{eq:res_bd}
\|R(is,A)\|=
O(|s|^{-\alpha}), \quad |s|\to0,
\end{equation}
for some $\alpha\ge1$.
Then 
\begin{equation}\label{eq:Mart}
\|AT(t)\|=O\left(\frac{\log(t)^{1/\alpha}}{t^{1/\alpha}}\right),\quad t\to\infty,
\end{equation}
and if $X$ is a Hilbert space then 
$$\|AT(t)\|=O(t^{-1/\alpha}),\quad t\to\infty.$$
\end{thm}

Returning now to our particular case, define 
$$Y=\Big\{x_0\in X:\lim_{t\to\infty}x(t)\mbox{ exists}\Big\}.$$
Furthermore, let $S\in\B(X)$ denote the right-shift operator on $X$, defined by $Sx=(0,x_1,x_2,\dots)$, and let $M\in\B(X)$ be defined by $Mx=(A_1A_0^{-1}x_k)$ for all $x=(x_k)_{k\ge1}\in X$.

\begin{thm}\label{thm:asymp}
Let $m\in\NN$ and $1\le p\le\infty$. Suppose that $\sigma(A_0)\subseteq\CC_-$, that $0\in\Omega_\phi^+\subseteq\CC_-\cup\{0\}$ and that $\phi$ is completely monotonic.  Then
\begin{equation}\label{eq:Y}
Y=\left\{x_0\in X:\bigg\|\frac1n\sum_{k=1}^n\phi(0)^kS^kMx_0\bigg\|\to0\mbox{ as }n\to\infty\right\},
\end{equation}
and $Y=X$ if and only if $1<p<\infty$. Moreover, $\|x(t)\|\to0$ as $t\to\infty$ whenever $x_0\in Y$. 
If $x_0\in Y$ is such that 
\begin{equation}\label{eq:range}
\bigg\|\frac1n\sum_{k=1}^n\phi(0)^kS^kMx_0\bigg\|=O(n^{-1})\mbox{ as }n\to\infty,
\end{equation}
then 
\begin{equation}\label{eq:rate}
\|x(t)\|=O\Bigg(\frac{\log (t)^{|1/2-1/p|}}{t^{1/2}}\Bigg),\quad t\to\infty.
\end{equation}
Finally, for $1\le p\le\infty$ and all $x_0\in X$ we have 
\begin{equation}\label{eq:rate2}
\|\dot{x}(t)\|=O\Bigg(\frac{\log (t)^{|1/2-1/p|}}{t^{1/2}}\Bigg),\quad t\to\infty.
\end{equation}
\end{thm}

\begin{proof}
We merely outline the main steps of the proof since it follows the same argument as given in \cite[Section~4]{PauSei17a}, with the slight simplification that now $\Ker A=\{0\}$ for $1\le p\le\infty$. Note first that boundedness of the semigroup $\T$ follows from Theorem~\ref{thm:bdd}.
The key idea of the proof is to establish first that $Y$ coincides with the closure of $\Ran A$, which can then be characterised as in \eqref{eq:Y} using results from classical ergodic theory. The next step is to show that the condition in \eqref{eq:range} characterises elements which lie not just in the closure of $\Ran A$ but in $\Ran A$ itself. We then combine Proposition~\ref{prp:res}, Lemma~\ref{lem:res2} and  Theorem~\ref{thm:mlog} to obtain rates of convergence for $\|AT(t)\|$ as $t\to\infty$. The rates in \eqref{eq:rate} and \eqref{eq:rate2} then follow by applying the Riesz-Thorin interpolation theorem to improve the exponent of the logarithm when $p\ne2$.
\end{proof}

\begin{rem}\label{rem:opt}
The decay rate obtained for $\|AT(t)\|$ as $t\to\infty$ in the above proof, and hence the rates in \eqref{eq:rate} and \eqref{eq:rate2}, are sharp except perhaps for the logarithmic factor. Indeed, since $\|R(is,A)\|\asymp|s|^{-2}$ as $|s|\to0$ it follows from \cite[Corollary~6.11]{BatChi16} that $\|AT(t)\|\ge ct^{-1/2}$, $t\ge1$, for some $c>0$; see also \cite[Remark~4.11]{PauSei17a}.
\end{rem}

\subsection{Improved rates for totally monotonic functions}
It is known that the rates in \eqref{eq:rate} and \eqref{eq:rate2} are sharp if $p=2$, and it was conjectured in \cite[Remark~4.14(b)]{PauSei17a} that the logarithmic factors in these estimates can in fact be dropped for all values of $p$. We now show that this conjecture to be correct provided we replace the monotonicity assumption on $\phi$ by a slightly stronger assumption. To this end we introduce the concept of \emph{total monotonicity}. We say that a rational function $\phi$ is \emph{totally monotonic} if there exist $\eps>0$ and a non-negative sequence $(a_n)_{n\ge0}\in\ell^1(\ZZ_+
)$ such that the function $\psi_\eps$ defined by $\psi_\eps(\lambda)=\phi(\eps^{-1}(\lambda-1))$ satisfies
$$\psi_\eps(\lambda)=\sum_{n=0}^\infty\frac{a_n}{\lambda^{n+1}},\quad |\lambda|\ge1.$$
In particular, any pole $\lambda$ of $\phi$ satisfies $|\eps\lambda+1|<1$.  The function $\psi_\eps$ may be  viewed as a unilateral $Z$-transform of the sequence $(a_n)_{n\ge0}$, which is a discrete analogue of the Laplace transform. Suppose that $\phi$ is a rational function which decays to zero at infinity, has poles $\xi_1,\dots,\xi_N\in\CC_-$ and admits the partial fraction decomposition
$$\phi(\lambda)=\sum_{j=1}^N\sum_{k=1}^{n_j}\frac{A_{j,k}}{(\lambda-\xi_j)^k},\quad \lambda\in\rho(A_0),$$
for some $n_1,\dots,n_N\in\NN$ and $A_{j,k}\in\CC$. Then for $\eps>0$ sufficiently small to ensure that $|\eps\xi_j+1|<1$ for $1\le j\le N$ we have
$$\psi_\eps(\lambda)=\sum_{j=1}^N\sum_{k=1}^{n_j}\left(\frac{\eps}{\lambda}\right)^k\frac{A_{j,k}}{(1-\lambda^{-1}(\eps\xi_j+1))^k}=\sum_{n=0}^\infty\frac{a_{\eps,n}}{\lambda^{n+1}},\quad |\lambda|\ge1,$$
where
$$a_{\eps,n}=\sum_{j=1}^N\sum_{k=1}^{n_j}A_{j,k}\binom{n}{k-1}\eps^k(\eps\xi_j+1)^{n-k+1},\quad n\ge0.$$
The sequence $(a_{\eps,n})_{n\ge0}$ is summable, and hence the function $\phi$ is totally monotonic if and only if there exists a sufficiently small $\eps>0$ such that $a_{\eps,n}\ge0$ for all $n\ge0$.

For a given number $\eps>0$ let us say that a function $\phi$ is $\eps$-\emph{totally monotonic} if it is totally monotonic and we can choose the given value of $\eps$ in the definition of total monotonicity. Observe that if a function $\phi$ is $\eps$-totally monotonic for some $\eps>0$, then $\phi$ is also $\eps_0$-totally monotonic for all $\eps_0\in(0,\eps)$. Indeed, suppose that $\phi$ is $\eps$-totally monotonic and let $(a_n)_{n\ge0}$ be as in the definition. Given $\eps_0\in(0,\eps)$ let $\beta=\eps_0/\eps$. Then for $|\lambda|\ge1$ we have $|\beta^{-1}(\lambda-1+\beta)|\ge1$ and hence
$$\psi_{\eps_0}(\lambda)=\psi_\eps\left(\frac{\lambda-1+\beta}{\beta}\right)=\sum_{n=0}^\infty\frac{a_n\beta^{n+1}}{(\lambda-1+\beta)^{n+1}}=\sum_{n=0}^\infty\frac{b_n}{\lambda^{n+1}},\quad |\lambda|\ge1,$$
where 
$$b_n=\sum_{k=0}^na_k\binom{n}{k}\beta^{k+1}(1-\beta)^{n-k},\quad n\ge0.$$
In particular, the sequence $(b_n)_{n\ge0}$ is non-negative and summable with $\sum_{n=0}^\infty b_n=\sum_{n=0}^\infty a_n$, so $\phi$ is $\eps_0$-totally monotonic.
Furthermore, the class of totally monotonic functions is closed under taking sums, products and under multiplication by non-negative scalars. Any totally monotonic function is completely monotonic. Indeed, if $\phi$ is totally monotonic we may consider the function $g\colon(0,\infty)\to\CC$ defined by
\begin{equation*}\label{eq:g}
g(t)=\frac{1}{\eps}e^{-t/\eps}\sum_{n=0}^\infty\frac{a_n}{n!}\left(\frac t\eps\right)^n,\quad t>0,
\end{equation*}
where $\eps>0$ and $(a_n)_{n\ge0}$ are as in the definition of total monotonicity. Then $g$ is non-negative and integrable with $\int_0^\infty g(t)\,\dd t=\sum_{n=0}^\infty a_n$, and moreover
$$\phi(\lambda)=\int_0^\infty e^{-\lambda t}g(t)\,\dd t,\quad \lambda\ge0,$$
so by Bernstein's theorem $\phi$ is completely monotonic. Alternatively, one may simply differentiate term-by-term the expression for $\phi(\lambda)=\psi_\eps(1+\eps\lambda)$ when $\lambda>0$ to verify directly that $\phi$ is completely monotonic.

\begin{ex}\label{ex:tot_mon}
\begin{enumerate}[(a)]
\item If $\zeta>0$ then the function $\phi$ defined by $\phi(\lambda)=\zeta(\lambda+\zeta)^{-1}$ is totally monotonic. Indeed, for $\eps\in(0,\zeta^{-1})$ we have
$$\psi_\eps(\lambda)=\frac{\eps\zeta}{\lambda-1+\eps\zeta}=\sum_{n=0}^\infty\frac{\eps\zeta(1-\eps\zeta)^n}{\lambda^{n+1}},\quad |\lambda|\ge1,$$ 
so we may set $a_n=\eps\zeta(1-\eps\zeta)^n$, $n\ge0$. It follows that functions $\phi$ of the form considered in \eqref{eq:phi_prod} are also totally monotonic.
\item The function $\phi$ considered in \eqref{eq:phi_pair} is totally monotonic if and only if $a>c$. In particular, for $a=c$ the function is completely monotonic but not totally monotonic.
\item As in Example~\ref{ex:comp_mon}(c)  we may combine the examples in (a) and (b) to obtain total monotonicity for a larger class of rational functions $\phi$.
\end{enumerate}
\end{ex}

We now sharpen Theorem~\ref{thm:asymp} by showing that the logarithmic factors in \eqref{eq:rate} and \eqref{eq:rate2} and are not needed when the characteristic function is totally monotonic. We moreover assume that the characteristic function is of the form $\phi(\lambda)=p_0(0)/p_0(\lambda)$, where $p_0$ is the characteristic polynomial of the matrix $A_0$. In view of the fact that $\phi$ is of this form in all of the examples we have considered, and indeed many examples which arise naturally in applications, there is no serious loss of generality here. The improved rate is achieved by appealing to a result of Dungey \cite[Theorem~1.2]{Dun08}, which has previously been used in  \cite{PauSei18} for a similar purpose.

\begin{thm}\label{thm:no_log}
Let $m\in\NN$ and $1\le p\le\infty$. Suppose that $\sigma(A_0)\subseteq\CC_-$ and that $0\in\Omega_\phi^+\subseteq\CC_-\cup\{0\}$. Assume further that $\phi$ is totally monotonic and that $\phi(\lambda)=p_0(0)/p_0(\lambda)$ for all $\lambda\in\rho(A_0)$, where $p_0$ is the characteristic polynomial of $A_0$.  If $x_0\in X$ is such that 
\eqref{eq:range} holds 
then $\|x(t)\|=O(t^{-1/2})$  as $t\to\infty$. Moreover, for $1\le p\le\infty$ and all $x_0\in X$ we have $\|\dot{x}(t)\|=O(t^{-1/2})$  as $t\to\infty.$ Both of these rates are sharp.
\end{thm}

\begin{proof}
By Theorem~\ref{thm:asymp} and the preceding discussion it suffices to prove that $\|AT(t)\|=O(t^{-1/2})$ as $t\to\infty$, noting that optimality follows from Remark~\ref{rem:opt}. Let $\eps>0$ and $(a_n)_{n\ge0}$ be as in the definition of total monotonicity, and let $\eps_0\in(0,\eps)$. We may assume that $\eps$ is sufficiently small to ensure that $|\eps\lambda+1|<1$ for all $\lambda\in\sigma(A_0)$. We shall show that the operator $B=\eps A+I$ is power-bounded. It then follows from the implication (ii)$\implies$(v) of \cite[Theorem~1.2]{Dun08} (with $n=1$) that the semigroup $(T_0(t))_{t\ge0}$ generated by $\eps_0A$ satisfies $\|AT_0(t)\|=O(t^{-1/2})$ as $t\to\infty$, from which the result follows immediately. Note that $Bx=(B_0x_1,B_0x_2+B_1x_1,B_0x_3+B_1x_2,\dotsc)$ for $x=(x_k)_{k\ge1}\in X,$ where $B_0=\eps A_0+I$ and $B_1=\eps A_1$. For $|\eps\lambda+1|>1$ we have
$$|\phi(\lambda)|=|\psi_\eps(\eps\lambda+1)|=\bigg|\sum_{n=0}^\infty\frac{a_n}{(\eps\lambda+1)^{n+1}}\bigg|<\sum_{n=0}^\infty a_n=\phi(0)=1.$$
Since $|\eps\lambda+1|<1$ for all $\lambda\in\sigma(A_0)$, it follows from Proposition~\ref{prp:spec} and the spectral mapping theorem that $\sigma(B)=\eps\sigma(A)+1\subseteq\{\lambda\in\CC:|\lambda|\le1\}$. A straightforward calculation shows that 
$$B_1R(\lambda,B_0)B_1=\psi_\eps(\lambda)B_1$$
for all $\lambda\in\rho(B_0)$. Let us write $R_\lambda$ for $R(\lambda,B_0)$ when $\lambda\in\rho(B_0)$. Then from \eqref{eq:res} we have
$$R(\lambda,B)x=\left(R_\lambda x_k+R_\lambda B_1R_\lambda\sum_{\ell=0}^{k-2}\psi_\eps(\lambda)^\ell x_{k-\ell-1}\right)_{k\ge1}$$
for all $x\in X$ and $|\lambda|>1$. Let $\Gamma$ be a circle centred on the origin of radius greater than 1, oriented counterclockwise. Then for $x\in X$ and $n\ge1$ we have
$$B^n x=\frac{1}{2\pi i}\oint_\Gamma \lambda^n R(\lambda,B)x\,\dd\lambda$$
by the Dunford-Schwartz functional calculus for bounded linear operators, and hence
\begin{equation}\label{eq:Bn}
B^nx=(B_0^nx_k)_{k\ge1}+\left(\frac{1}{2\pi i}\sum_{\ell=0}^{k-2}\oint_\Gamma \lambda^n\psi_\eps(\lambda)^\ell R_\lambda B_1R_\lambda x_{k-\ell-1}\,\dd\lambda \right)_{k\ge1}.
\end{equation}
By our assumption that $\phi(\lambda)=p_0(0)/p_0(\lambda)$ for $\lambda\in\rho(A_0)$,  where $p_0(\lambda)=\det(\lambda-A_0)$, we may find $m\times m$ matrices $C_0,\dots,C_{2m-2}$ such that
$$R_\lambda B_1 R_\lambda=\frac{1}{p_0(\eps^{-1}(\lambda-1))^2}\sum_{j=0}^{2m-2}C_j\lambda^j,\quad \lambda\in\rho(B_0).$$
It follows using \eqref{eq:Bn} that
\begin{equation}\label{eq:norm_Bn}
\|B^n\|\le\|B_0^n\|+\frac{1}{2\pi}\sum_{j=0}^{2m-2}\frac{\|C_j\|}{|p_0(0)|^2}\sum_{\ell=0}^\infty\left|\oint_\Gamma\lambda^{n+j}\psi_\eps(\lambda)^{\ell+2}\,\dd\lambda\right|,\quad n\ge1.
\end{equation}
For each $\ell\ge0$ we may find a non-negative summable sequence $(a_n^{(\ell)})_{n\ge0}$ such that
$$\psi_\eps(\lambda)^\ell=\sum_{n=0}^\infty\frac{a_n^{(\ell)}}{\lambda^{n+\ell}},\quad |\lambda|\ge1.$$
Note that $a_n^{(\ell+1)}=\sum_{k=0}^na_ka_{n-k}^{(\ell)}$ for all $n,\ell\ge0$ and that  $\sum_{n=0}^\infty a_n^{(\ell)}=\psi_\eps(1)^\ell=\phi(0)^\ell=1$ for all $\ell\ge0$. It follows from Cauchy's residue theorem that
\begin{equation}\label{eq:Cauchy}
\frac{1}{2\pi}\sum_{\ell=0}^\infty\left|\oint_\Gamma\lambda^{n}\psi_\eps(\lambda)^{\ell+2}\,\dd\lambda\right|=\sum_{\ell=0}^{n-1}a^{(\ell+2)}_{n-1-\ell},\quad n\ge1.
\end{equation}
We now prove by induction that 
\begin{equation}\label{eq:ind}
\sum_{\ell=0}^{n-1}a^{(\ell+2)}_{n-1-\ell}\le \sum_{\ell=0}^{n-1}a^{(2)}_{n-1-\ell}\le1,\quad n\ge1.
\end{equation}
Note that the second inequality is straightforward, so we focus on the first. The claim holds trivially for $n=1$. Suppose it holds for positive integers less than or equal to some $n$. Then 
$$\begin{aligned}
\sum_{\ell=0}^{n}a^{(\ell+2)}_{n-\ell}&=a_n^{(2)}+\sum_{\ell=0}^{n-1}a^{(\ell+3)}_{n-1-\ell}=a_n^{(2)}+\sum_{\ell=0}^{n-1}\sum_{k=0}^{n-1-\ell}a_ka_{n-1-\ell-k}^{(\ell+2)}\\&=a_n^{(2)}+\sum_{k=0}^{n-1}a_k\sum_{\ell=0}^{n-1-k}a_{n-1-\ell-k}^{(\ell+2)}\le a_n^{(2)}+\sum_{k=0}^{n-1}a_k\sum_{\ell=0}^{n-1-k}a_{n-1-\ell-k}^{(2)}\\&\le a_n^{(2)}+\sum_{k=0}^{n-1}a_k\sum_{\ell=0}^{n-1}a_{n-1-\ell}^{(2)}\le a_n^{(2)}+\sum_{\ell=0}^{n-1}a_{n-1-\ell}^{(2)}=\sum_{\ell=0}^{n}a_{n-\ell}^{(2)},
\end{aligned}$$
where the first inequality follows from the inductive hypothesis, the second from the fact that the sequence $\smash{(a_n^{(2)})_{n\ge0}}$ is non-negative and the third from the fact that $\sum_{n=0}^\infty a_n=1$. Thus \eqref{eq:ind} holds. Using this in \eqref{eq:Cauchy} shows that the second term on the right-hand side of \eqref{eq:norm_Bn} is uniformly bounded for $n\ge1$. Since $\sup_{n\ge1}\|B_0^n\|<\infty$ as a consequence of the spectral radius formula, we deduce that $B$ is power-bounded, so the proof is complete.
\end{proof}

\subsection{The two-sided case}\label{sec:two}

 In this section we consider the two-sided case \eqref{eq:sys_two}, presenting the corresponding versions of the results obtained in the previous sections. We begin by rewriting the system in the form of an abstract Cauchy problem as in \eqref{eq:ACP}, where now  $X=\ell^p(\ZZ,\CC^m)$ for $m\in\NN$ and $1\le p\le\infty$, and 
\begin{equation}\label{eq:A_two}
Ax=(A_0x_k+A_1x_{k-1})_{k\in\ZZ},\quad x=(x_k)_{k\in\ZZ}\in X.
\end{equation}
 As noted in Section~\ref{sec:spec}, the results concerning the spectrum and the resolvent of the operator $A$ in the one-sided case are direct analogues of the results proved in \cite[Section~2]{PauSei17a} for the two-sided case, and we refer the reader to \cite{PauSei17a} for these results. The main difference is that the role of $\Omega_\phi^+$ is now played by the set $\Omega_\phi=\{\lambda\in\rho(A_0):|\phi(\lambda)|=1\}$.

On the other hand, the use of monotonicity conditions is new even in the two-sided case, so we summarise the main statements in the following theorem.

\begin{thm}\label{thm:two}
Let $m\in\NN$ and $1\le p\le\infty$. Suppose that $\sigma(A_0)\subseteq\CC_-$, that $0\in\Omega_\phi\subseteq\CC_-\cup\{0\}$ and that $\phi$ is completely monotonic.  Then the semigroup $\T$ generated by $A$ is bounded, and moreover
$$\|AT(t)\|=O\Bigg(\frac{\log (t)^{|1/2-1/p|}}{t^{1/2}}\Bigg),\quad t\to\infty.$$
Furthermore, if $\phi$ is totally monotonic and satisfies $\phi(\lambda)=p_0(0)/p_0(\lambda)$ for all $\lambda\in\rho(A_0)$, where $p_0$ is the characteristic polynomial of $A_0$, then $\|AT(t)\|=O(t^{-1/2})$ as $t\to\infty$. The latter rate is optimal.
\end{thm}

The proof of this result uses exactly the same ideas as were used in Theorems~\ref{thm:bdd}, \ref{thm:asymp} and \ref{thm:no_log} above, so we omit it. Moreover, as in Theorems~\ref{thm:asymp} and \ref{thm:no_log} above it is possible to deduce from Theorem~\ref{thm:two} statements about rates of convergence to equilibrium of suitable orbits and about rates of decay of the derivatives of orbits. The corresponding statements in the two-sided case are slightly more involved because when $p=\infty$ orbits may converge to non-zero steady states. We do not make these statements precise here,  instead referring the reader to \cite[Theorem~4.3]{PauSei17a}, and to Remark~\ref{rem:approx} and Section~\ref{sec:platoon} below.

\section{Uniform behaviour of finite systems}\label{sec:finite}

In this  section we use our results for infinite systems to derive uniform estimates for large but finite systems.  We wish to obtain uniform asymptotic estimates for the semigroups $\TN$ generated by the operators $A_N$, $N\ge2$, which are obtained as truncations of the operator $A$ in one of our infinite models of the form \eqref{eq:sys_two} and \eqref{eq:sys_one}. These uniform estimates can be used to study the asymptotic properties of solutions $x(t)=T_N(t)x_0$, $t\ge0$, of the abstract Cauchy problem 
\begin{equation*}\label{eq:ACP_N}
\left\{\begin{aligned}
\dot{x}(t)&=A_Nx(t),\quad t\ge0,\\
{x}(0)&=x_0\in X_N,
\end{aligned}\right.
\end{equation*}
as the size $N$ grows large.

Consider first the one-sided model \eqref{eq:sys_one}.  For $1\le p\le\infty$ and $m\in\NN$, the natural truncations to consider involve the operators $A_N$, $N\ge2$,  acting on the spaces $X_N=\ell^p_N(\CC^m)$ of $\CC^m$-valued sequences of length $N$ endowed with the $p$-norm by $$A_Nx=(A_0x_1,A_0x_2+A_1x_1,\dots,A_0x_N+A_1x_{N-1}),\quad x\in X_N.$$
This case is rather simple and may be dealt with at once. Indeed, for $N\ge2$ let us write $J_N\colon X_N\to X$ for the embedding operator defined by $J_Nx=y$, $x\in X$, where $y_n=x_n$ for $1\le n\le N$ and $y_n=0$ for $n>N$, and let $P_N\colon X\to X_N$ denote the operator which truncates an infinite sequence after its $N$-th entry. Then $P_NJ_N$ is the identity operator on $X_N$, and we have $\|J_N\|=\|P_N\|=1$. Observing that $A_N=P_NAJ_N$ and $T_N(t)=P_NT(t)J_N$, where $\TN$ is the $C_0$-semigroup generated by $A_N$, we deduce that $\|A_NT_N(t)\|\le \|AT(t)\|$, $t\ge0$, for all $N\ge2$. In particular, if the characteristic function
$\phi$ determined by $A_0$ and $A_1$ is completely monotonic, if $\sigma(A_0)\subseteq\mathbb{C}_-$ and if $0\in\Omega_\phi^+\subseteq \mathbb{C}_-\cup \{0\}$, then by (the proof of) Theorem~\ref{thm:asymp}
$$
\sup_{N\ge2}\|A_NT_N(t)\| = O\Bigg(\frac{\log (t)^{|1/2-1/p|}}{t^{1/2}}\Bigg),\quad t\to\infty.
$$

The two-sided model  \eqref{eq:sys_two} is much more delicate. Given any integer $N\ge2$ the most natural truncations of the infinite model are the \emph{circulant truncations}, in which the operator $A$ defined in \eqref{eq:A_two} is replaced by the operator $A_N$ acting on $X_N=\ell^p_N(\CC^m)$ by 
$$A_Nx=(A_0x_1+A_1x_N,A_0x_2+A_1x_1,\dots,A_0x_N+A_1x_{N-1}),\quad x\in X_N.$$
We begin with an abstract result in the spirit of  \cite{MaNa16}.

\begin{prp}\label{prp:uniform}
Let  $S$ be a non-empty countable set and let $X_N$, $N\in S$, be  complex Banach spaces. For each $N\in S$ let $\TN$ be a  $C_0$-semigroup on $X_N$ whose generator $A_N$ satisfies $\sigma(A_N)\cap i\RR\subseteq\{0\}$. Suppose moreover that   $\sup_{N\in S}\|A_N\|<\infty$, that
\begin{equation}\label{eq:unif_bd}
\sup_{N\in S}\sup_{t\ge0}\|T_N(t)\|<\infty
\end{equation}
and that 
\begin{equation}\label{eq:res_bd_approx}
\sup_{N\in S}\|R(is,A_N)\|\le
\begin{cases}
C|s|^{-\alpha},& 0<|s|\le1,\\
C,&|s|\ge1,
\end{cases}
\end{equation}
for some $C>0$ and $\alpha\ge1$.
Then 
\begin{equation}\label{eq:Mart_approx}
\sup_{N\in S}\|A_NT_N(t) \|=O\left(\frac{\log (t)^{1/\alpha}}{t^{1/\alpha}}\right),\quad t\to\infty,
\end{equation}
and if each of the spaces $X_N$, $N\in S$, is a Hilbert space then 
$$\sup_{N\in S}\|A_NT_N(t)\|=O\big(t^{-1/\alpha}\big),\quad t\to\infty.$$
\end{prp}

\begin{proof}
Let $X$ denote the $\ell^2$-direct sum $\bigoplus_{N\in S} X_N$. Then $X$ is a Banach space, and $X$ is a Hilbert space if each of the spaces $X_N$, $N\in S$, is a Hilbert space. If $B_N\in\B(X_N)$, $N\in S$,  and $\sup_{N\in S}\|B_N\|<\infty$, then $B=\bigoplus_{N\in S}B_N$ is a bounded linear operator on $X$ with norm $\|B\|=\sup_{N\in S}\|B_N\|$. For $t\ge0$ let $T(t)=\bigoplus_{N\in S}T_N(t)$. Then $\T$ is a $C_0$-semigroup on $X$ with generator $A=\bigoplus_{N\in S}A_N$, and 
by \eqref{eq:unif_bd} the semigroup is bounded.
It follows from  \eqref{eq:res_bd_approx} that  $\sigma(A)\cap i\RR\subseteq\{0\}$ and  
 that \eqref{eq:res_bd} holds. Hence  Theorem~\ref{thm:mlog} implies \eqref{eq:Mart}, and that the logarithm may be omitted if $X_N$ is a Hilbert space for every $N\in S$.
\end{proof}

\begin{rem}
Proposition~\ref{prp:uniform} can be extended and generalised in a number of ways. In particular, by using more general versions of Theorem~\ref{thm:mlog} one may obtain variants in which the operator $A=\bigoplus_{N\in S}A_N$ appearing in the above proof is allowed to be unbounded; see also \cite{MaNa16}.
\end{rem}

Our next theorem provides asymptotic estimates for the semigroups generated by the circulant truncations $A_N$, $N\ge2$,  which are uniform in $N$.  Here we again let $\Omega_\phi=\{\lambda\in\rho(A_0):|\phi(\lambda)|=1\}$.

\begin{thm}\label{thm:uniform}
Let $m\in\NN$ and $1\le p\le\infty$. Suppose that $\sigma(A_0)\subseteq\CC_-$, that $0\in\Omega_\phi\subseteq\CC_-\cup\{0\}$ and that $\phi$ is completely monotonic. For $N\ge2$ let $A_N$ be the circulant truncation defined as above on $X_N=\ell^p_N(\CC^m)$, and let $\TN$ be the $C_0$-semigroup generated by $A_N$. Then 
\begin{equation}\label{eq:decay_N}
\sup_{N\ge2}\|A_NT_N(t)\|=O\left(\frac{\log(t)^{|1/2-1/p|}}{t^{1/2}}\right),\quad t\to\infty.
\end{equation}
Furthermore, if $\phi$ is totally monotonic and $\phi(\lambda)=p_0(0)/p_0(\lambda)$ for all $\lambda\in\rho(A_0)$, where $p_0$ is the characteristic polynomial of $A_0$, then 
\begin{equation}\label{eq:no_log}
\sup_{N\ge2}\|A_NT_N(t)\|=O\big(t^{-1/2}\big),\quad t\to\infty.
\end{equation}
\end{thm}

\begin{proof}
Let $N\ge2$ and suppose that $x,y\in X_N$ and $\lambda\in\CC$ are such that $(\lambda-A_N)x=y$. Then $(\lambda-A_0)x_k=y_k+A_1x_{k-1}$, $1\le k\le N$. If $\lambda\in\rho(A_0)$, we therefore obtain 
\begin{equation}\label{eq:res_N}
x_k=R_\lambda y_k+R_\lambda A_1x_{k-1},\quad 1\le k\le N,
\end{equation}
where $R_\lambda=R(\lambda,A_0)$ for brevity. Applying $A_1$ to both sides and using \eqref{eq:char} we find after a straightforward inductive argument that 
$$A_1x_k=\frac{1}{1-\phi(\lambda)^N}\sum_{\ell=0}^{N-1}\phi(\lambda)^\ell A_1R_\lambda y_{k-\ell},\quad 1\le k\le N,\ \lambda\in\CC\setminus\Omega_\phi.$$
It follows that $\sigma(A_N)\subseteq\Omega_\phi$, and from \eqref{eq:res_N} we obtain the expression
\begin{equation}\label{eq:res_exp_N}
R(\lambda,A_N)x=\left(R_\lambda x_k+\frac{1}{1-\phi(\lambda)^N}\sum_{\ell=0}^{N-1}\phi(\lambda)^\ell R_\lambda A_1R_\lambda x_{k-\ell}\right)_{1\le k\le N}
\end{equation}
for all $\lambda\in\CC\setminus\Omega_\phi $ and $x\in X_N$. The argument used in \cite[Theorem~3.1]{PauSei17a} with only very small modifications shows that the semigroup $\TN$ is bounded provided that  \eqref{eq:phi_bound} holds and 
\begin{equation}\label{eq:res_bd_N}
\sup_{n\in\NN}\sup_{\lambda>0}\frac{\lambda^{n+1}}{n!}\sum_{\ell=0}^{N-1}\bigg|\frac{\dd^n}{\dd\lambda^n}\frac{\phi(\lambda)^\ell}{1-\phi(\lambda)^N}\bigg|<\infty.
\end{equation}
Moreover, if there is an upper bound in \eqref{eq:res_bd_N} which is independent of $N$, then we will in fact have shown that 
\begin{equation}\label{eq:sg_bd_N}
\sup_{N\ge2}\sup_{t\ge0}\|T_N(t)\|<\infty.
\end{equation}
As shown in the proof of Theorem~\ref{thm:bdd} above, condition \eqref{eq:phi_bound} holds by our assumption that $\phi$ is completely monotonic. Since $\phi(\lambda)\in(0,1)$ for all $\lambda>0$, the same assumption also implies that the restriction to $(0,\infty)$ of the function $\lambda\mapsto\phi(\lambda)^\ell(1-\phi(\lambda)^N)^{-1}$ is completely monotonic for every $\ell\ge0$ and $N\ge2$. Arguing once again as in the proof of Theorem~\ref{thm:bdd} we find that \eqref{eq:res_bd_N} reduces to \eqref{eq:mon_bd}, which is not only independent of $N$ but also satisfied, as was shown in the same proof. Thus \eqref{eq:sg_bd_N} holds. It moreover follows from \eqref{eq:res_exp_N} that
$$\|R(is,A_N)\|\le \|R_{is}\|+\frac{\|R_{is} A_1 R_{is}\|}{1-|\phi(is)|},\quad s\in\RR\setminus\{0\},$$
and hence by \cite[Lemma~2.6]{PauSei17a}, which is the two-sided version of Lemma~\ref{lem:res} above, and by Lemma~\ref{lem:res2} we see that  \eqref{eq:res_bd_approx} holds for $\alpha=2$ and some $C>0$.
It now follows from Proposition~\ref{prp:uniform} and the Riesz-Thorin interpolation theorem applied to the product semigroup appearing in the proof of Proposition~\ref{prp:uniform} that \eqref{eq:decay_N} holds. The final statement follows from the argument used in the proof of Theorem~\ref{thm:no_log} applied to the product semigroup.
\end{proof}

\begin{rem}\phantomsection\label{rem:approx}
 Recall from \cite{PauSei17a} that by differentiating \eqref{eq:char} and setting $\lambda=0$ we obtain $-A_1A_0^{-2} A_1=\phi'(0)A_1$. Hence if $\phi'(0)\ne0$ then $A_1A_0^{-1}$ maps $\Ran(A_0^{-1}A_1)$ bijectively onto $\Ran A_1$. We denote the inverse of this isomorphism by $L$. Note that $\phi'(0)\ne0$ when $\phi$ is completely monotonic, in which case we moreover have $\phi(0)=1$. It is then straightforward to show that in the setting of Theorem~\ref{thm:uniform} we have $\Ker A_N=\{(y,\dots,y)\in X_N:y\in\Ran(A_0^{-1}A_1)\}$ and that the projection $P_N\colon X_N\to X_N$ onto $\Ker A_N$ along $\Ran  A_N$ is given by $P_Nx=(LQx,\dots ,LQx)$, $x\in X_N$, where $Q\colon X_N\to \CC^m$ is given by
$$Qx=\frac1N\sum_{k=1}^NA_1A_0^{-1}x_k,\quad x\in X_N.$$
For any $N\ge2$ and $x\in X_N$ we have 
$$T_N(t)x-P_Nx=T_N(t)(x-P_Nx)=A_NT_N(t)y,\quad t\ge0,$$
for some $y\in X_N$, so Theorem~\ref{thm:uniform} implies that $\|T_N(t)x-P_Nx\|\to0$ as $t\to\infty$ for all $x\in X_N$, but it moreover provides a rate of convergence which is uniform in $N$ subject to the condition that the norms of the preimages $y\in X_N$ are uniformly bounded.
\end{rem}

We conclude this section with a simple example showing that the uniform rates of convergence obtained in \eqref{thm:uniform} are in general best possible.

\begin{ex}
Suppose that $m=1$ and $1\le p\le \infty$, and let  $A_0=-1$ and $A_1=1$. In this case, which arises in the so-called \emph{robot-rendezvous model} \cite{FeiFra12}, all the assumptions of Theorem~\ref{thm:uniform} are satisfied, and the characteristic function $\phi$ is given by $\phi(\lambda)=(\lambda+1)^{-1}$, $\lambda\in\CC\setminus\{-1\}$. In particular, $\phi$ is totally monotonic and of the form $\phi(\lambda)=p_0(0)/p_0(\lambda)$,  $\lambda\in\CC\setminus\{-1\}$, where $p_0$ is the characteristic polynomial of the `matrix' $A_0$, so by Theorem~\ref{thm:uniform} we obtain the uniform estimate \eqref{eq:no_log}.
Now for each $N\ge2$ the matrix $A_N$ is a circulant matrix whose spectrum satisfies $\sigma(A_N)=\{\lambda\in\CC:(\lambda+1)^N=1\}$, and hence the spectral mapping theorem implies that
$$\|A_NT_N(t)\|\ge\max_{1\le k<N}\big|e^{2\pi ik/N}-1\big|\exp{\big(t\left(\cos(2\pi k/N)-1\right)\big)},\quad t\ge0.$$
Considering only the eigenvalue corresponding to $k=1$ and choosing $t_N=N^2$, it follows from a simple estimate that
$$\|A_NT_N(t_N)\|\ge ct_N^{-1/2},\quad N\ge2,$$ 
for some constant $c>0$, so \eqref{eq:no_log} is sharp in this case even though for each individual $N\ge2$ the truncated model converges to its equilibrium at an exponential rate. In fact, by Remark~\ref{rem:approx} the agents in the truncated models always converge to the centroid of their initial positions.

In the corresponding one-sided model,  a simple calculation shows that for $p\in\{1,\infty\}$ we have
$$\|A_NT_N(t)\|\ge\frac{t^{N-1}}{(N-1)!}e^{-t}, \quad N\ge2,\ t\ge0,$$
and setting $t_N=N$ it follows from Stirling's formula that  
$$\|A_NT_N(t_N)\|\ge ct_N^{-1/2},\quad N\ge2,$$ 
for some constant $c>0$. The proof of Theorem~\ref{thm:no_log}, or alternatively a simple direct calculation, shows that the infinite system satisfies $\|AT(t)\|=O(t^{-1/2})$ as $t\to\infty$. It follows that the method described at the beginning of the section, of obtaining a uniform decay rate for the truncated systems by comparing them to the corresponding infinite system, is optimal in this case, at least for $p\in\{1,\infty\}$.
 \end{ex}

\section{Asymptotic behaviour in the platoon model}\label{sec:platoon}

In this final section we apply the results of the previous sections to the important example of the two-sided \emph{platoon model}, which has previously been studied for instance in \cite{PloSch11, SwaHed96,ZwaFir13}, and by the authors in \cite{PauSei17a, PauSei18}. In the platoon model we consider countably many agents, which we think of as vehicles and which are indexed by $k\in\ZZ$. The state vector of agent $k$  takes to form
$$x_k(t)= \left( \begin{array}{c}
y_k(t) \\
v_k(t)-v\\
a_k(t)  \end{array} \right),\quad k\in\ZZ,\; t\ge0,$$
where $y_k(t)=d_k-d_k(t)$ denotes the discrepancy between the target separation $d_k$ between agents $k$ and $k-1$ and their actual distance $d_k(t)$  at time $t$, $v_k(t)$ is the velocity of agent $k$ at time $t$, $v$ the target velocity of the entire platoon, and $a_k(t)$ is the acceleration of agent $k$ at time $t$. We take the system to evolve according to \eqref{eq:sys_two}, where $m=3$ and the matrices $A_0, A_1$ are given by 
$$A_0
= \left( \begin{array}{ccc}
 0&1& 0 \\
 0 & 0 & 1\\
-\alpha_0 & -\alpha_1 & -\alpha_2  \end{array} \right)
\quad\mbox{and}\quad
A_1= \left( \begin{array}{ccc}
0 & -1 & 0 \\
0&0 & 0\\
0&0 & 0 \end{array} \right).$$
Here the constants $\alpha_0,\alpha_1,\alpha_2\in\CC$ are thought of as control  parameters. They capture how the agents adjust their acceleration in response to their current state vector; see \cite[Section~5]{PauSei17a} for further details. Since $A_1$ has rank one, condition \eqref{eq:char} is satisfied and a simple calculation shows that the characteristic function of the system is given by $\phi(\lambda)=\alpha_0/p_0(\lambda)$, $\lambda\in\rho(A_0)$, where $p_0(\lambda)=\lambda^3+\alpha_2\lambda^2+\alpha_1\lambda+\alpha_0$ is the characteristic polynomial of $A_0$. For the matrix $A_0$ to satisfy $\sigma(A_0)\subseteq\CC_-$ we need the coefficients $\alpha_0,\alpha_1,\alpha_2$ to satisfy certain conditions, and in particular we need $\alpha_0\ne0$. As remarked in \cite[Remark~5.2(b)]{PauSei17a}, it is possible for the semigroup $\T$ generated by the system operator $A$ to be bounded but non-contractive, even when $p=2$. Since $\phi'(0)=-\alpha_1/\alpha_0$, for the condition $\phi'(0)\ne0$ appearing in Remark~\ref{rem:approx} to hold, and hence for there to be any hope of $\phi$ being completely monotonic, we moreover need $\alpha_1\ne0$. 

We obtain the following generalisation of \cite[Theorem~5.1]{PauSei17a} and \cite[Corollary~5.1]{PauSei18}.  The key improvement here is that we obtain boundedness of the underlying semigroup and (almost) sharp rates of decay for much larger classes of characteristic functions than could previously be handled. We define 
$$Y=\Big\{x_0\in X:\lim_{t\to\infty}x(t)\mbox{ exists}\Big\}$$
and we let $S$ denote the right-shift operator on $\ell^1(\ZZ)$ and on $\ell^\infty(\ZZ)$. Statements (b) through (d) of the result are consequences of Theorem~\ref{thm:two}, while part (a) follows from the same arguments as in the proof of \cite[Theorem~5.1]{PauSei17a}.

\begin{thm}
Let $1\le p\le\infty$ and consider the platoon model. Suppose that $\alpha_1,\alpha_2,\alpha_3\in\CC$ are such that the characteristic function $\phi$ is completely monotonic.   
\begin{enumerate}[(a)]
\item[\textup{(a)}] We have $Y=X$ if and only if $1< p<\infty$. More specifically:
\begin{enumerate}
\item[\textup{(i)}] If $1<p<\infty$ then $Y=X$ and $x(t)\to0$ for all $x_0\in X$.

\item[\textup{(ii)}] If $p=1$ and $x_0\in X$ then $x_0\in Y$ if and only if the vector $y_0=(y_k(0))_{k\in\ZZ}\in\ell^1(\ZZ)$ of initial deviations is such that
\begin{equation}
\label{lim_plat_finite}
\bigg\|\frac{1}{n}\sum_{k=1}^nS^{k}y_0\bigg\|_{\ell^1(\ZZ)}\!\!\to0,\quad n\to\infty,
\end{equation}
and if this  holds then $x(t)\to0$ as $t\to\infty$. 
\item[\textup{(iii)}] If $p=\infty$ and $x_0\in X$ then $x_0\in Y$ if and only if there exists $c\in\CC$ such that for $y=(\dotsc,c,c,c,\dotsc)$ we have
\begin{equation}
\label{lim_plat_infty}
\bigg\|\frac{1}{n}\sum_{k=1}^nS^{k}y_0-y\bigg\|_{\ell^\infty(\ZZ)}\!\!\to0,\quad n\to\infty,
\end{equation}
and if this holds then $x(t)\to z$ as $t\to\infty$, where 
\begin{equation}
z=\left(\dotsc,\begin{pmatrix}
c\\ -\alpha_0 c/\alpha_1\\ 0
    \end{pmatrix},\begin{pmatrix}
c\\ -\alpha_0 c/\alpha_1\\ 0
    \end{pmatrix},\begin{pmatrix}
c\\ -\alpha_0 c/\alpha_1\\ 0
    \end{pmatrix},\dotsc\right).
\end{equation}
\end{enumerate}

\item[\textup{(b)}] 
\begin{enumerate}
\item[\textup{(i)}] If $1\le p<\infty$ and the decay in \eqref{lim_plat_finite} is like $O(n^{-1})$ as $n\to\infty$  then
\begin{equation}\label{eq:plat_rate}
\|x(t)\|=O\left(\frac{\log( t)^{|1/2-1/p|}}{t^{1/2}}\right),\quad t\to\infty.
\end{equation}
\item[\textup{(i)}] If $p=\infty$ and the decay in \eqref{lim_plat_infty} is like $O(n^{-1})$ as $n\to\infty$ then
\begin{equation}\label{eq:plat_rate2}
\|x(t)-z\|=O\bigg(\frac{\log (t)^{1/2}}{t^{1/2}}\bigg),\quad t\to\infty.
\end{equation}
\end{enumerate}
\item[\textup{(c)}] For $1\le p\le\infty$ and all  $x_0\in X$ we have
\begin{equation}\label{eq:plat_rate3}
\|\dot{x}(t)\|=O\left(\frac{\log (t)^{|1/2-1/p|}}{t^{1/2}}\right),\quad t\to\infty.
\end{equation}
\item[\textup{(d)}] If the characteristic is totally monotonic, then we may omit the logarithmic factors in \eqref{eq:plat_rate}, \eqref{eq:plat_rate2} and \eqref{eq:plat_rate3} even when $p\ne2$, and the resulting rates are optimal.
 \end{enumerate}
\end{thm}

We may also use the platoon model to illustrate our results on uniform rates of decay for large but finite cyclic systems, as considered in Section~\ref{sec:finite}. Note that in this case  the projection $P_N$ onto $\Ker A_N$ along $\Ran A_N$ for $N\ge2$ considered in Remark~\ref{rem:approx} is given by
$$P_Nx=\left(\left(\begin{array}{c}
c_x \\
-\alpha_0 c_x/\alpha_1\\
0  \end{array} \right),\dots,\left(\begin{array}{c}
c_x \\
-\alpha_0 c_x/\alpha_1\\
0  \end{array} \right)\right),\quad x\in X_N,$$
where  $c_x\in\CC$ denotes the centroid of the first components of $x\in X_N$.

\begin{ex}
Let $\zeta_1,\zeta_2,\zeta_3>0$ and let $\alpha_0=-\zeta_1\zeta_2\zeta_3$, $\alpha_1=\zeta_1\zeta_2+\zeta_2\zeta_3+\zeta_3\zeta_1$, $\alpha_2=\zeta_1+\zeta_2+\zeta_3$. Then $\sigma(A_0)=\{-\zeta_1,-\zeta_2, -\zeta_3\}$ and by Example~\ref{ex:tot_mon}(a) the characteristic function $\phi$ is totally monotonic in this case. By Theorem~\ref{thm:two} we have $\|AT(t)\|=O(t^{-1/2})$ as $t\to\infty$, and by Theorem~\ref{thm:uniform} the circulant truncations satisfy the uniform estimate~\eqref{eq:no_log}. Using Remark~\ref{rem:approx} we may deduce a uniform rates of convergence to the steady states of the growing circulant truncations, which are obtained by applying the operators $P_N$, $N\ge2$, to the initial data.
\end{ex}

\begin{ex}
Let $a,b,c>0$ and let $\alpha_0=(a^2+b^2)c$, $\alpha_1=a^2+b^2+2ac$, $\alpha_2=2a+c$. Then $\sigma(A_0)=\{-c,-a\pm ib\}$ and by Examples \ref{ex:comp_mon}(a) and  \ref{ex:tot_mon}(a) the characteristic function $\phi$ is completely monotonic if and only if $a\ge c$ and it is totally monotonic if and only if $a>c$. In either case the resolvent growth parameter is $n_\phi=2$. By Theorem~\ref{thm:two} we have 
\begin{equation}\label{eq:platoon}
\|AT(t)\|=O\Bigg(\frac{\log (t)^{|1/2-1/p|}}{t^{1/2}}\Bigg),\quad t\to\infty,
\end{equation}
for $a= c$ and $\|AT(t)\|=O(t^{-1/2})$ as $t\to\infty$ when $a>c$. We expect that the logarithmic term in \eqref{eq:platoon} can be omitted even when $a=c$ and $p\ne2$. Once again we may study the uniform asymptotic behaviour of large circulant truncations using  Theorem~\ref{thm:uniform} and  Remark~\ref{rem:approx}.
\end{ex}

\bibliographystyle{plain}

\end{document}